%% file: main.tex
\documentclass[12pt]{amsart}
\usepackage[utf8]{inputenc}
\usepackage{amssymb}
\usepackage{amsmath}
\usepackage{amsthm}
\usepackage[T1]{fontenc}
\usepackage{tikz}
\usetikzlibrary{decorations.markings}
\usepackage{tkz-graph}
\usepackage{mathtools}
\usepackage{wasysym}
\usepackage{verbatim}
\usepackage{comment}
\usepackage{graphicx}
\usepackage{color}
\usepackage{xcolor}
\usepackage[margin=1.15in]{geometry}
\usepackage{caption}
\usepackage{subcaption}
\usepackage{geometry,tikz-cd,fancyhdr,wrapfig,enumitem,thm-restate}
\usepackage{hyperref}
\hypersetup{colorlinks=true,citecolor=purple,linkcolor=purple}
\usepackage[noabbrev,capitalize,nameinlink]{cleveref}
\usepackage{style}
\usepackage[style=alphabetic,doi=false,isbn=false,url=false,eprint=false]{biblatex}
\begin{document}

\title{Acyclic Orientations and the Chromatic Polynomial of Signed Graphs}
\author[J. Gao]{Jiyang Gao}
\email{jgao@math.harvard.edu}
\date{\today}

\maketitle
\vspace{-0.5cm}
\begin{abstract}
We present a new correspondence between acyclic orientations and coloring of a signed graph (symmetric graph). Goodall et al. introduced a bivariate chromatic polynomial $\chi_G(k,l)$ that counts the number of signed colorings using colors $0,\pm1,\dots,\pm k$ along with $l-1$ symmetric colors $0_1,\dots,0_{l-1}$. We show that the evaluation of the bivariate chromatic polynomial $|\chi_G(-1,2)|$ is equal to the number of acyclic orientations of the signed graph modulo the equivalence relation generated by swapping sources and sinks. We present three proofs of this fact, a proof using toric hyperplane arrangements, a proof using deletion-contraction, and a direct proof.
\end{abstract}

\input{1-Intro}
\input{2-Prelim}
\input{3-Toric}
\input{4-Graphical}
\input{5-Delcont}
\input{6-Directly-1}
\input{6-Directly-2}

\printbibliography

\end{document}

%% file: 1-Intro.tex
\section{Introduction}

In 1973, Richard Stanley \cite{Stanl_Acyclicorientations_73} found a simple yet astonishing connection between the chromatic polynomial and acyclic orientations of a graph. Let $\chi_G$ be the chromatic polynomial of $G$. He showed that $\chi_G(-1)$ is, up to a sign, equal to the number of acyclic orientations of $G$.

Later on, Greene and Zaslavsky \cite{Green_InterpretationWhitney_83} found a simple yet different connection between these two objects. Namely, they showed that the linear coefficient of $\chi_G$ is, up to a sign, equal to the number of acyclic orientations of $G$ with a unique fixed sink. Furthermore, this number equals an evaluation of the Tutte polynomial $T_G(1,0)$. This number also counts equivalence classes of acyclic orientations under the \emph{flip} move. Specifically, two acyclic orientations are flip equivalent to each other if we can convert a sink (or source) to a source (or sink) by reversing all edges adjacent to that vertex. This flip move was introduced by Pretzel \cite{Pretz_reorientinggraphs_86}, and further studied by many past research \cite{Propp_Latticestructure_02,Devel_Toricpartial_16}.

In this work, we generalize these results to symmetric graphs. Symmetric graphs are a variation of signed graphs introduced by Harary \cite{Harar_notionbalance_53} and further studied by Zaslavsky \cite{Zasla_Signedgraphs_82,Zasla_Signedgraphcol_82,Zasla_Chromaticinvariants_82}. In the literature, they are also called \emph{double covering graphs} in \cite{Zasla_Signedgraphs_82} or \emph{skew-symmetric graphs} in \cite{Baben_AcyclicBidirected_06,Tutte_AntisymmetricalDigraphs_67}. Simply speaking, symmetric graphs are graphs with an involution $\iota$ such that $uv$ is an edge if and only if $\iota(u)\iota(v)$ is an edge. Notions such as the chromatic polynomial and acyclic orientations are all generalizable to symmetric graphs (details in \cref{sec:prelim}). Our main result is that, the number of equivalence classes of acyclic orientations modulo certain flip moves, is equal to a specialization of the bivariate chromatic polynomial of symmetric graph introduced by \cite{Gooda_Tuttedichromate_21}.

There are several reasons why this result was not previously discovered. First, this number is not a simple evaluation of the Tutte polynomial of the graphical matroid of the symmetric graph (also known as the \emph{frame matroid} in \cite{Zasla_Signedgraphs_82}). In fact, it is an evaluation only if we consider the trivariate Tutte polynomial introduced recently by \cite{Gooda_Tuttedichromate_21}. Secondly, while the bivariate chromatic polynomial satisfies a simple deletion-contraction reccurence, there is no such obvious recurrence for the number of acyclic orientations modulo flip moves, which is different from the case of ordinary graphs. Therefore, a simple induction proof does not exist.

In fact, both colorings and acyclic orientations are deeply connected to the geometry of toric hyperplane arrangements, which were studied recently in \cite{Novik_SyzygiesOriented_00,Ehren_AffineToric_09,Deshp_Facecounting_19}. We will first introduce the basics of toric hyperplane arrangements in \cref{sec:toric}, and then in \cref{sec:graphical} we define the graphical toric hyperplane arrangement of $G$ introduced by \cite{Green_InterpretationWhitney_83}. Analogous to the results in \cite{Devel_Toricpartial_16}, we show that there is a bijection between acyclic orientations of $G$ modulo flip moves and connected chambers in the graphical toric hyperplane arrangement of $G$. From there, we prove our main results in two ways: a finite field method from \cite{Ehren_AffineToric_09} in \cref{sec:graphical}, and a deletion-restriction proof in \cref{sec:delcont}. Finally, we will include a direct proof of the result without the use of geometry in \cref{sec:directly}. The proof uses the reciprocity idea of Stanley's \cite{Stanl_Acyclicorientations_73}, and a direct induction.

%% file: 2-Prelim.tex
\section{Preliminaries: Symmetric graphs}\label{sec:prelim}

\subsection{Signed Graphs and Symmetric Graphs}

A \emph{signed graph} $\Sigma=(\Gamma,\sigma)$ is a pair of an unsigned graph $\Gamma=(V,E)$ equipped with a function $\sigma$ which labels each edge (except an half edge) in $\Gamma$ positive or negative. The underlying unsigned graph $\Gamma$ may have not only links and loops but also \emph{half edges} (which has only one endpoint).

Although signed graph appears more often in literature, there is an alternative form of a signed graph which we will use below.

\begin{definition}
A \emph{symmetric graph} is a graph $G=(V,E)$ on vertices
\[V=\{v_{-n},v_{-(n-1)},\dots,v_0,v_1,\dots,v_{n}\}\]
such that if $(v_i,v_j)\in E$, then $(v_{-i},v_{-j})\in E$.
\end{definition}

There is a natural correspondence between signed graphs $\Sigma$ on $n$ vertices and symmetric graphs $G$ on $2n+1$ vertices. The correspondence is summarized in \cref{tab:signedsymmetric}. 
\begin{table}[ht]
\centering
\begin{tabular}{|c|c|}
\hline
Signed Graph $\Sigma$ & Symmetric Graph $G$\\ \hline\hline
Vertex $u_i$ & Vertex pair $\{v_i,v_{-i}\}$\\ \hline
Positive Link $u_iu_j$ & Edges $v_iv_j$ and $v_{-i}v_{-j}$\\ \hline
Negative Link $u_iu_j$ & Edges $v_iv_{-j}$ and $v_{-i}v_{j}$\\ \hline
Positive Loop $u_i$ & Loops $v_i$ and $v_{-i}$\\ \hline
Negative Loop $u_i$ & Edge $v_iv_{-i}$\\ \hline
Half edge $u_i$ & Edges $v_0v_i$ and $v_0v_{-i}$\\ \hline
\end{tabular}
\caption{Correspondence between a signed graph and a symmetric graph}
\label{tab:signedsymmetric}
\end{table}

\begin{example}
\cref{fig:1} shows a signed graph and the corresponding symmetric graph. The signed graph on the left contains a positive link, a negative link, and a half edge.
\begin{figure}[ht!]
\centering
\begin{subfigure}[c]{.4\linewidth}
\centering
\begin{tikzpicture}
    \SetFancyGraph
    \SetGraphUnit{3}
    \Vertex[L=$u_1$,Lpos=below]{1}
    \EA[L=$u_2$,Lpos=below](1){2}
    \EA[empty=true,unit=1.5](2){3}
    \Edge[style={bend left},label=$+$,labelstyle=above](1)(2)
    \Edge[style={bend right},label=$-$,labelstyle=below](1)(2)
    \Edge(2)(3)
\end{tikzpicture}
\caption{Signed Graph}\label{fig:1a}
\end{subfigure}
\begin{subfigure}[c]{.4\linewidth}
\centering
\begin{tikzpicture};
    \SetFancyGraph
    \SetGraphUnit{1.5}
    \Vertex[L=$v_1$,Lpos=left]{1}
    \EA[L=$v_0$,Lpos=left](1){0}
    \NO[L=$v_2$,Lpos=above](0){2}
    \SO[L=$v_{-2}$,Lpos=below](0){-2}
    \EA[L=$v_{-1}$,Lpos=right](0){-1}
    \Edges(1,2,-1,-2,1)
    \Edges(2,0,-2)
\end{tikzpicture}
\caption{Symmetric Graph}\label{fig:1b}
\end{subfigure}
\caption{Example of Signed Graph and Symmetric Graph}\label{fig:1}
\end{figure}
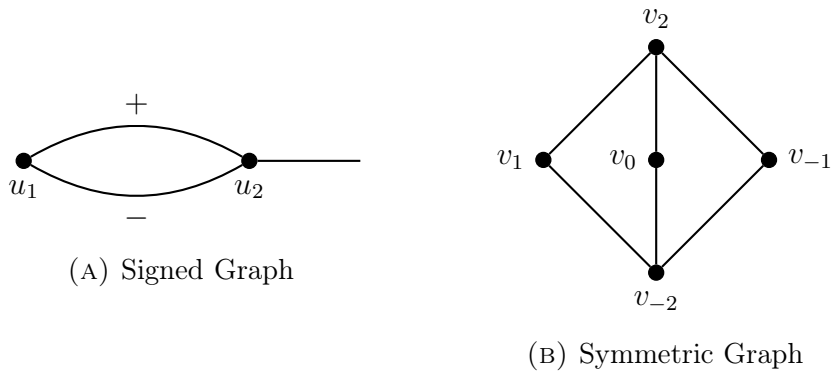
\end{example}

\begin{definition}\label{def:weaklyconn}
A symmetric graph is \emph{simple} if it is simple as a graph. It is \emph{weakly connected} if vertex $v_i$ is connected to $v_{-i}$ for all $i\neq 0$.
\end{definition}

Without explicitly saying, we will assume all symmetric graphs are simple in the remaining paper. That is, it contains neither loops nor multiple edges.

\subsection{Colorings of symmetric graphs}
In \cite{Gooda_Tuttedichromate_21}, Goodall et al. introduced a bivariate $(k,l)$-coloring of a symmetric graph $G$ using colors $\{0,\pm 1,\dots, \pm k\}$ along with an extra set of colors $\{0_1,\dots,0_{l-1}\}$ where we rule that $0_i=-0_i$.

\begin{definition}
Let $k,l\geq 1$ be integers. A \emph{proper $(k,l)$-coloring} of a symmetric graph $G$ is an assignment $f:V\to \{0,\pm 1,\dots, \pm k\}\cup\{0_1,\dots,0_{l-1}\}$ such that
\begin{itemize}
    \item $f(v_0)=0$;
    \item $f(v_i)=-f(v_{-i})$ for any $i\neq 0$;
    \item $f(v_i)\neq f(v_j)$ for any edge $v_iv_j\in E$.
\end{itemize}
\end{definition}

Goodall et al. \cite[Theorem~6.5]{Gooda_Tuttedichromate_21} showed that the number of proper $(k,l)$-colorings of $G$ is also a polynomial in $k$ and $l$, called the \emph{bivariate chromatic polynomial} $\chi_G(k,l)$. When $l=1$, a proper $(k,1)$-coloring is also called a \emph{proper $k$-coloring}, which was first introduced by Zaslavsky in \cite{Zasla_Signedgraphcol_82}.

\begin{remark}
It is an interesting question to find combinatorial interpretations for integer evaluations of the bivariate chromatic polynomial $\chi_G(k,l)$ when $k\leq 0$. The case $l=1$ is well-studied, and later \cref{thm:main} will give an example for $l=2$. This question is still open even for the case $l=0$, which is called \emph{proper non-zero $k$-colorings} by Zaslavsky in \cite{Zasla_Signedgraphcol_82}.
\end{remark}

\begin{example}
We calculate the bivariate chromatic polynomial $\chi_G(k,l)$ of the symmetric graph $G$ in \cref{fig:1b}. We consider two cases based on the coloring of $v_2$.
\begin{itemize}
    \item If $f(v_2)=0_i$ for some $i$, then $v_1$ can be colored any color other than $0_i$. There are $2k+l-1$ choices for $f(v_1)$ and $l-1$ choices for $i$.
    \item If $f(v_2)=j$ for some $r\in\{\pm 1,\dots,\pm k\}$, then $v_1$ can be colored any color other than $\pm j$. There are $2k+l-2$ choices for $f(v_1)$ and $2k$ choices for $j$.
\end{itemize}
Combining the two cases,
\begin{equation}\label{eq:chro}
    \chi_G(k,l)=(2k+l-1)(l-1)+(2k+l-2)2k.
\end{equation}
\end{example}

\subsection{Acyclic orientations of symmetric graphs}

\begin{definition}
A \emph{symmetric acyclic orientations} of symmetric graph $G$ is an acyclic orientation of $G$ such that if $v_i\to v_j$ then $v_{-j}\to v_{-i}$.
\end{definition}

The acyclic orientation of a symmetric graph (or signed graph) was studied by Zaslavsky in \cite{Zasla_Signedgraphcol_82,Zasla_OrientationSigned_91}. Analogous to Richard Stanley's celebrated result \cite{Stanl_Acyclicorientations_73} where the number of acyclic orientations of an ordinary graph is equal to its chromatic polynomial evaluated at $-1$, Greene and Zaslavsky in \cite{Green_InterpretationWhitney_83} proved the following result.

\begin{theorem}[{\cite[Theorem~9.1]{Green_InterpretationWhitney_83}}]\label{thm:zas}
For a symmetric graph $G$ with $2n+1$ vertices, the number of symmetric acyclic orientations of $G$ is equal to $(-1)^n\chi_G(-1,1)$.
\end{theorem}

\begin{definition}
In a symmetric acyclic orientation $\omega$ of symmetric graph $G$, if $v_i (i\neq 0)$ is a source (sink) vertex of $\omega$, then $v_{-i}$ is a sink (source). If $v_iv_{-i}$ is not an edge, then we call $(v_i,v_{-i})$ a \emph{source-sink pair}. Reversing the direction of all edges adjacent to the source-sink pair $(v_i,v_{-i})$, we get another symmetric acyclic orientation $\omega'$ of $G$, and we say $\omega$ and $\omega'$ differ by a \emph{flip}, denoted $\omega\sim\omega'$. The transitive closure of the flip operation generates an equivalence relation on the set of all symmetric acyclic orientation $\text{Acyc}(G)$ denoted by $\sim$.
\end{definition}

The following is our main theorem.

\begin{theorem}\label{thm:main}
For a weakly connected symmetric graph $G$ with $2n+1$ vertices, we have
\[|\text{Acyc}(G)/\sim|=(-1)^n\chi_G(-1,2).\]
\end{theorem}

\begin{example}
We consider the symmetric graph $G$ in \cref{fig:1b}. On the one hand, its bivariate chromatic polynomial is computed in \cref{eq:chro} and we can evaluate it
\[\chi_G(-1,2)=3.\]
On the other hand, $G$ has exactly $3$ equivalence classes of symmetric acyclic orientations as shown in \cref{fig:2}.
\begin{figure}[ht!]
\centering
\begin{subfigure}[c]{.3\linewidth}
\centering
\begin{tikzpicture}[scale=0.8]
    \SetFancyGraph
    \SetGraphUnit{1.5}
    \begin{scope}
    \tikzset{EdgeStyle/.append style=->--}
    \Vertex[L=$v_1$,Lpos=left]{1}
    \EA[L=$v_0$,Lpos=left](1){0}
    \NO[L=$v_2$,Lpos=above](0){2}
    \SO[L=$v_{-2}$,Lpos=below](0){-2}
    \EA[L=$v_{-1}$,Lpos=right](0){-1}
    \Edges(2,1,-2)
    \Edges(2,0,-2)
    \Edges(2,-1,-2)

    \SO[L=$v_2$,Lpos=above,unit=4](-2){2a}
    \SO[L=$v_0$,Lpos=left](2a){0a}
    \WE[L=$v_1$,Lpos=left](0a){1a}
    \SO[L=$v_{-2}$,Lpos=below](0a){-2a}
    \EA[L=$v_{-1}$,Lpos=right](0a){-1a}
    \Edges(-2a,1a,2a)
    \Edges(-2a,0a,2a)
    \Edges(-2a,-1a,2a)
    \end{scope}
    
    \SO[empty=true,unit=1](-2){X}
    \NO[empty=true,unit=1](2a){Y}
    \Edge[style=<->,label=flip $v_2/v_{-2}$](X)(Y)
\end{tikzpicture}
\caption{Class 1}
\end{subfigure}
\begin{subfigure}[c]{.3\linewidth}
\centering
\begin{tikzpicture}[scale=0.8]
    \SetFancyGraph
    \SetGraphUnit{1.5}
    \begin{scope}
    \tikzset{EdgeStyle/.append style=->--}
    \Vertex[L=$v_1$,Lpos=left]{1}
    \EA[L=$v_0$,Lpos=left](1){0}
    \NO[L=$v_2$,Lpos=above](0){2}
    \SO[L=$v_{-2}$,Lpos=below](0){-2}
    \EA[L=$v_{-1}$,Lpos=right](0){-1}
    \Edges(1,2,-1)
    \Edges(2,0,-2)
    \Edges(1,-2,-1)

    \SO[L=$v_2$,Lpos=above,unit=4](-2){2a}
    \SO[L=$v_0$,Lpos=left](2a){0a}
    \WE[L=$v_1$,Lpos=left](0a){1a}
    \SO[L=$v_{-2}$,Lpos=below](0a){-2a}
    \EA[L=$v_{-1}$,Lpos=right](0a){-1a}
    \Edges(-1a,2a,1a)
    \Edges(2a,0a,-2a)
    \Edges(-1a,-2a,1a)
    \end{scope}
    
    \SO[empty=true,unit=1](-2){X}
    \NO[empty=true,unit=1](2a){Y}
    \Edge[style=<->,label=flip $v_1/v_{-1}$](X)(Y)
\end{tikzpicture}
\caption{Class 2}
\end{subfigure}
\begin{subfigure}[c]{.3\linewidth}
\centering
\begin{tikzpicture}[scale=0.8]
    \SetFancyGraph
    \SetGraphUnit{1.5}
    \begin{scope}
    \tikzset{EdgeStyle/.append style=->--}
    \Vertex[L=$v_1$,Lpos=left]{1}
    \EA[L=$v_0$,Lpos=left](1){0}
    \NO[L=$v_2$,Lpos=above](0){2}
    \SO[L=$v_{-2}$,Lpos=below](0){-2}
    \EA[L=$v_{-1}$,Lpos=right](0){-1}
    \Edges(1,2,-1)
    \Edges(-2,0,2)
    \Edges(1,-2,-1)

    \SO[L=$v_2$,Lpos=above,unit=4](-2){2a}
    \SO[L=$v_0$,Lpos=left](2a){0a}
    \WE[L=$v_1$,Lpos=left](0a){1a}
    \SO[L=$v_{-2}$,Lpos=below](0a){-2a}
    \EA[L=$v_{-1}$,Lpos=right](0a){-1a}
    \Edges(-1a,2a,1a)
    \Edges(-2a,0a,2a)
    \Edges(-1a,-2a,1a)
    \end{scope}
    
    \SO[empty=true,unit=1](-2){X}
    \NO[empty=true,unit=1](2a){Y}
    \Edge[style=<->,label=flip $v_1/v_{-1}$](X)(Y)
\end{tikzpicture}
\caption{Class 3}
\end{subfigure}
\caption{Equivalence classes of symmetric acyclic orientations of $G$ in \cref{fig:1b}}\label{fig:2}
\end{figure}
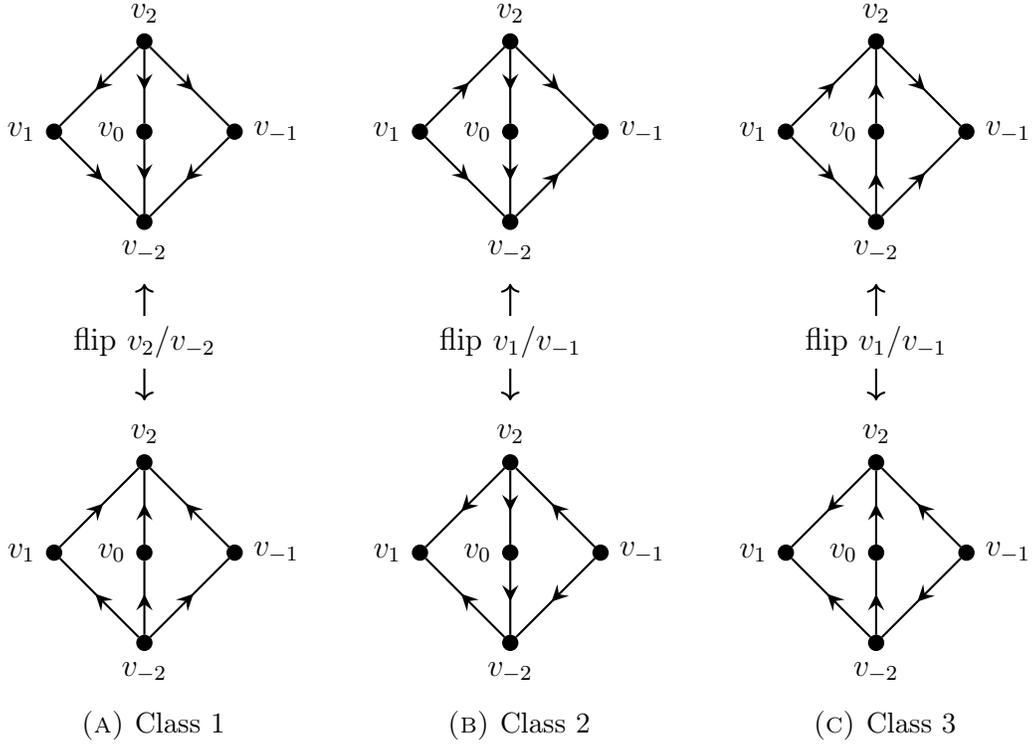
\end{example}

We will present three proofs for \cref{thm:main} in the following sections. The first two proofs relate both numbers to the number of chambers in a graphical toric hyperplane arrangement corresponding to $G$, while the last proof uses a direct induction.

%% file: 3-Toric.tex
\section{Preliminaries: Toric Hyperplane Arrangements}\label{sec:toric}

\begin{definition}
Let $\cH=\{H_1,\dots,H_m\}$ be the union of (rational) hyperplanes in $\R^n$ where in $H_i:\vec{h}_i\cdot \vec{x}=b_i$, all coordinates $\vec{h}_i$ and $b_i$ are rational. The \emph{toric hyperplane arrangement} $\cH_\text{tor}$ is the image of $\cH$ under the projection map $\pi:\R^n\to\R^n/\Z^n$. Its \emph{affine hyperplane arrangement} is the preimage $\cH_\text{aff}=\pi^{-1}(\cH_\text{tor})$.
\end{definition}

In simple terms, toric hyperplane arrangements are projections of hyperplane arrangements onto a tori. It have been studied extensively by many previous research, for example \cite{Novik_SyzygiesOriented_00,Ehren_AffineToric_09,Moci_TuttePolynomial_12,Devel_Toricpartial_16}. In this section, we will state some basic facts about toric hyperplane arrangements without proof. Most of the results can be found in \cite{Ehren_AffineToric_09}.

\begin{definition}
Let $V$ be a $k$-dimensional (rational) affine subspace of $\R^n$, where $V=\{\vec{x}:A\vec{x}=\vec{b}\}$ for some matrix $A$ and vector $\vec{b}$ with rational entries. Then the image $\pi(V)$ is called a \emph{$k$-dimensional toric subspace} of $\R^n/\Z^n$. An $(n-1)$-dimensional toric subspace is also called a \emph{toric hyperplane}.
\end{definition}

\begin{lemma}[{\cite[Lemma~3.1]{Ehren_AffineToric_09}}]
A $k$-dimensional toric subspace is homeomorphic to the $k$-dimensional torus $\R^k/\Z^k$. Moreover, the intersection of toric subspaces $V\cap W$ is a disjoint union of toric subspaces $V\cap W=\cup_{i=1}^r U_i$ for some integer $r$.
\end{lemma}

\begin{definition}
For a toric hyperplane arrangement $\cH_\text{tor}=\{H_1,\dots,H_m\}$, define the \emph{intersection lattice} $\cL(\cH)$ as the set of all connected components (toric subspaces) of all possible intersections of toric hyperplanes $H_{i_1}\cap \cdots \cap H_{i_r}$, ordered by inverse inclusion. It has a unique minimal element $\Hat{0}$ corresponding to the intersection of empty set, which is the entire space $\R^n/\Z^n$ itself. The \emph{toric characteristic polynomial} $\chi_\cH(q)$ is the characteristic polynomial of poset $\cL(\cH)$
\[\chi_\cH(q)=\sum_{x\in\cL(\cH)}\mu(\Hat{0},x)\cdot q^{\dim(x)}.\]
\end{definition}

\begin{example}
\cref{fig:3} shows an example of a toric hyperplane arrangement with hyperplanes $x_1=2x_2$, $2x_1=x_2$ and $x_1=x_2$. The corresponding intersection lattice $\cL(\cH)$ has toric characteristic polyomial $\chi_\cH(q)=q^2-3q+4$.
\begin{figure}[ht!]
\centering
\begin{subfigure}[c]{.4\linewidth}
\centering
\begin{tikzpicture}[scale=1.5]
    \draw[opacity=0.5,gray](-0.2,0) -- (2.2,0);
    \draw[opacity=0.5,gray](-0.2,2) -- (2.2,2);
    \draw[opacity=0.5,gray](0,-0.2) -- (0,2.2);
    \draw[opacity=0.5,gray](2,-0.2) -- (2,2.2);
    
    \draw[thick] (-0.2,-0.1) -- (2.2,1.1);
    \draw[thick] (-0.1,-0.2) -- (1.1,2.2);
    \draw[thick] (-0.2,0.9) -- (2.2,2.1);
    \draw[thick] (0.9,-0.2) -- (2.1,2.2);
    \draw[thick] (-0.1,-0.1) -- (2.1,2.1);
\end{tikzpicture}
\caption{Toric hyperplane arrangement}
\end{subfigure}
\begin{subfigure}[c]{.4\linewidth}
\centering
\begin{tikzpicture};
    \SetFancyGraph
    \SetGraphUnit{1.2}
    \Vertex[L=$\Hat{0}$,Lpos=below]{0}
    \NO[NoLabel](0){1}
    \NOEA[NoLabel](0){2}
    \NOWE[NoLabel](0){3}
    \NO[NoLabel](1){4}
    \NO[NoLabel](2){5}
    \NO[NoLabel](3){6}
    \Edges(2,0,3)
    \Edges(2,4,3)
    \Edges(2,6,3)
    \Edges(2,5,3)
    \Edges(0,1,4)
\end{tikzpicture}
\caption{Intersection Lattice $\cL(\cH)$}
\end{subfigure}
\caption{Example of a toric hyperplane arrangement of hyperplanes defined by $x_1=2x_2$, $2x_1=x_2$ and $x_1=x_2$}\label{fig:3}
\end{figure}
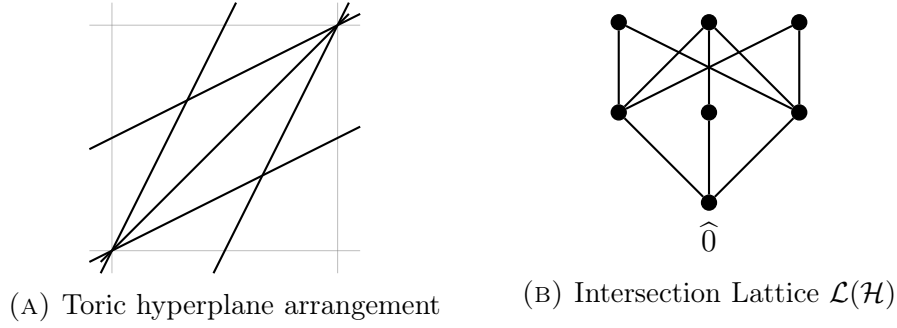
\end{example}

For a toric hyperplane arrangement $\cH_\text{tor}$, the connected components of $\R^n/\Z^n-\cH_\text{tor}$ are called \emph{toric chambers} of $\cH_\text{tor}$. The set of toric chambers is denoted as $\text{Cham}(\cH_\text{tor})$. Similarly, for its affine hyperplane arrangement $\cH_\text{aff}$, the connected components of $\R^n-\cH_\text{aff}$ are called \emph{affine chambers} of $\cH_\text{aff}$. The set of affine chambers is denoted as $\text{Cham}(\cH_\text{aff})$.

Recall that a hyperplane arrangement is \emph{essential} if the normal vectors of the hyperplanes span the entire space $\R^n$. When the toric arrangement $\cH_\text{tor}$ is essential, we can determine the number of toric chambers of $\cH_\text{tor}$ using the toric characteristic polynomial.

\begin{theorem}[{\cite[Theorem~3.6]{Ehren_AffineToric_09}}]\label{thm:chamandchar}
If $\cH$ is essential, then
\[|\text{Cham}(\cH_\text{tor})|=(-1)^n\chi_{\cH}(0).\]
\end{theorem}

%% file: 4-Graphical.tex
\section{A first proof: graphical toric hyperplane arrangements}\label{sec:graphical}

In order to prove the main theorem, we will show that both sides of \cref{thm:main} are equal to the number of chambers in a graphical toric hyperplane arrangement associated to $G$, which is defined below.

\begin{definition}
Given a symmetric graph $G=(V,E)$ with $2n+1$ vertices, define its \emph{graphical toric hyperplane arrangement} $\cB_\text{tor}(G)$ by the union of the following hyperplanes in $\R^n/\Z^n$.
\begin{multline*}
    \cB_\text{tor}(G)=\bigcup_{\substack{0<i<j\leq n\\(v_i,v_j)\in E}}\{x_i-x_j=0\}\cup\bigcup_{\substack{0<i<j\leq n\\(v_i,v_{-j})\in E}}\{x_i+x_j=0\}\cup\bigcup_{\substack{0<i\leq n\\(v_0,v_i)\in E}}\{x_i=0\}\\
    \cup\bigcup_{\substack{0<i\leq n\\(v_i,v_{-i})\in E}}\{x_i=0\}\cup\bigcup_{\substack{0<i\leq n\\(v_i,v_{-i})\in E}}\{2x_i=1\}.
\end{multline*}
Define $\cB_\text{aff}(G):=\pi^{-1}(\cB_\text{tor}(G))$ as the corresponding \emph{graphical affine hyperplane arrangement}.
\end{definition}

\begin{example}
For the symmetric graph $G$ in \cref{fig:1b}, \cref{fig:4} shows the affine hyperplane arrangement $\cB_\text{aff}(G)$ and toric hyperplane arrangement $\cB_\text{tor}(G)$ corresponding to $G$, with their chambers colored in different colors.
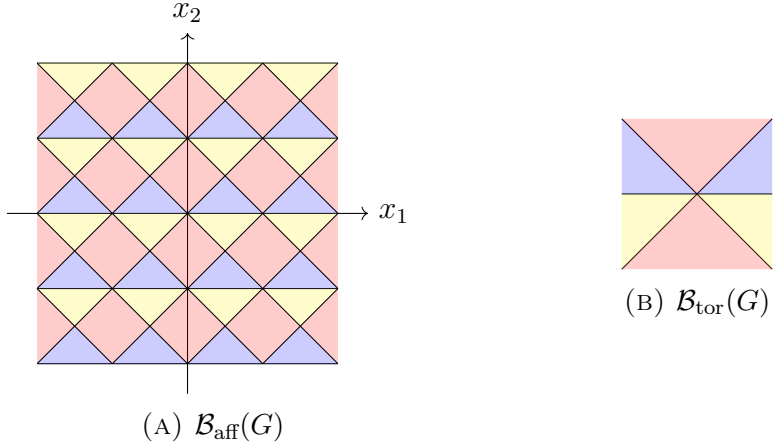
\begin{figure}[ht!]
\centering
\begin{subfigure}[c]{.4\linewidth}
\centering
\begin{tikzpicture}
    \draw[->] (-2.4,0)--(2.4,0) node[right]{$x_1$};
    \draw[->] (0,-2.4)--(0,2.4) node[above]{$x_2$};
    \draw (-2,1)--(2,1);
    \draw (-2,-1)--(2,-1);
    \draw (-2,2)--(2,2);
    \draw (-2,-2)--(2,-2);
    \foreach \x in {-2,-1,0,1}
    \foreach \y in {-2,-1,0,1}
    {\begin{scope}[shift={(\x,\y)}]
    \fill[red,opacity=0.2]  (0,0) -- (0.5,0.5) -- (0,1) -- cycle;
    \fill[red,opacity=0.2]  (1,0) -- (0.5,0.5) -- (1,1) -- cycle;
    \fill[yellow,opacity=0.2]  (0,1) -- (0.5,0.5) -- (1,1) -- cycle;
    \fill[blue,opacity=0.2]  (0,0) -- (0.5,0.5) -- (1,0) -- cycle;
    \draw (0,0)--(1,1);
    \draw (1,0)--(0,1);
    \end{scope}}
\end{tikzpicture}
\caption{$\cB_\text{aff}(G)$}
\end{subfigure}
\begin{subfigure}[c]{.4\linewidth}
\centering
\begin{tikzpicture}
    \draw (-1,-1)--(1,1);
    \draw (-1,1)--(1,-1);
    \draw (-1,0)--(1,0);
    \fill[red,opacity=0.2]  (0,0) -- (1,1) -- (-1,1) -- cycle;
    \fill[red,opacity=0.2]  (0,0) -- (-1,-1) -- (1,-1) -- cycle;
    \fill[yellow,opacity=0.2]  (0,0) -- (1,0) -- (1,-1) -- cycle;
    \fill[yellow,opacity=0.2]  (0,0) -- (-1,0) -- (-1,-1) -- cycle;
    \fill[blue,opacity=0.2]  (0,0) -- (1,0) -- (1,1) -- cycle;
    \fill[blue,opacity=0.2]  (0,0) -- (-1,0) -- (-1,1) -- cycle;
\end{tikzpicture}
\caption{$\cB_\text{tor}(G)$}
\end{subfigure}
\caption{Example of $\cB_\text{aff}(G)$ and $\cB_\text{tor}(G)$. The range of $\cB_\text{aff}(G)$ is $[-2,2]^2$, and the range of $\cB_\text{tor}(G)$ is $(-0.5,0.5]^2$. The coloring of different chambers is based on the projection map $\pi$.}
\label{fig:4}
\end{figure}
\end{example}

\begin{lemma}
The toric arrangement $\cB_\text{tor}(G)$ is essential if and only if $G$ is weakly connected.
\end{lemma}
\begin{proof}
If $G$ is not weakly connected, there exists index $i>0$ such that $v_i$ and $v_{-i}$ are not connected. Consider the connected component $S$ containing $v_i$. There are two facts about $S$.
\begin{itemize}
    \item $v_0\notin S$. Otherwise by symmetry of $G$, there is a path $v_i\to v_0\to v_{-i}$;
    \item If $v_j\in S$, then $v_{-j}\notin S$. This is because $v_i$ and $v_j$ are connected implies $v_{-j}$ and $v_{-i}$ are connected. 
\end{itemize}
Denote the connected component $S=\{v_i,v_{i_1},\dots,v_{i_p}\}\cup\{v_{-j_1},\dots,v_{-j_q}\}$ with distinct indices $i_1,\dots,i_p,j_1,\dots,j_q>0$. Then all normal vectors in $\cB_\text{tor}(G)$ are perpendicular the the vector
\[\vec{e}_S=\vec{e}_{i}+\vec{e}_{i_1}+\cdots+\vec{e}_{i_p}-\vec{e}_{j_1}-\cdots-\vec{e}_{j_q},\]
where $\vec{e}_i$ is the vector with a $1$ in the $i^\text{th}$ coordinate and $0$'s elsewhere.

If $\cB_\text{tor}(G)$ is not essential, then there exists a non-zero vector $\vec{w}=(w_1,\dots,w_n)$ which is perpendicular to all normal vectors. Consider a non-zero entry $w_i\neq 0$. We claim that $v_i$ and $v_{-i}$ are not connected in $G$. Otherwise there is a path from $v_i$ to $v_{-i}$. If we sum up the normal vectors associated to each edge along the path, we get $2\vec{e}_i$. Since $\vec{w}$ is perpendicular to all normal vectors, we have $w_i=\vec{w}\cdot \vec{e}_i=0$, which is a contradiction.
\end{proof}

\subsection{Chambers of $\cB_\text{tor}(G)$ and Acyclic Orientations of $G$}

A point $x$ in $\R^n/\Z^n$ does not have a well-defined coordinate $(x_1,\dots,x_n)$. However, if we translate every coordinate $x_i$ by an integer we can assume each $x_i$ lies in the interval $(-0.5,0.5]$, and this coordinate is well-defined. For each point $x=(x_1,x_2,\dots,x_n)\in \R^n/\Z^n-\cB_\text{tor}(G)$, we denote $x_0=0$ and $x_{-i}=-x_i$ for any $1\leq i\leq n$. Now we define an acyclic orientation $\omega(x)$ of $G$ as follows: for any $v_iv_j\in E(G)$ ($i,j\in[-n,n]$), since $x_i\neq x_j$, we have either
\begin{itemize}
    \item $x_i<x_j$, then we direct $v_j\to v_i$, or
    \item $x_i>x_j$, then we direct $v_i\to v_j$.
\end{itemize}
Notice that this orientation $\omega(x)$ is indeed symmetric since $x_i<x_j\iff x_{-i}>x_{-j}$. Denote this map $x\mapsto \omega(x)$ as $\beta_G:(\R^n/\Z^n-\cB_\text{tor}(G))\to\text{Acyc}(G)$.

\begin{theorem}\label{thm:chamandacyc}
The map $\beta_G$ induces an isomorphism $\Bar{\beta}_G$ from $\text{Cham}(\cB_\text{tor}(G))$ to $\text{Acyc}(G)/\sim$ as follows:
\[\begin{tikzcd}
    \R^n/\Z^n-\cB_\text{tor}(G) \arrow[r, "\beta_G"] \arrow[d, two heads]
    & \text{Acyc}(G) \arrow[d, two heads] \\
    \text{Cham}(\cB_\text{tor}(G)) \arrow[r, dashed, "\Bar{\beta}_G"] & \text{Acyc}(G)/\sim
\end{tikzcd}\]
In other words, two points $x,y$ lie in the same chamber in $\text{Cham}(\cB_\text{tor}(G))$ if and only if $\beta_G(x)\sim\beta_G(y)$.
\end{theorem}

\begin{proof}
\underline{Well-definition:} We need to show that if two points $x,y$ lie in the same chamber $c$, then $\beta_G(x)\sim \beta_G(y)$. There exists a path $\gamma$ in $c$ between $x$ and $y$. Since $c$ is open, we can assume $\gamma$ takes steps in coordinate directions only, and therefore reduce to the case where $x$ and $y$ differ in only one coordinate $x_i\neq y_i$. Since $\beta_G(x)$ only changes when passing through the boundary line $x_i=0.5$, one may assume $x_i=0.5-\varepsilon$ and $y_i=-0.5+\varepsilon$ for some arbitrarily small $\varepsilon>0$. In such case $\beta_G(x)$ and $\beta_G(y)$ differs by flipping the source-sink pair $(v_i,v_{-i})$, so $\beta_G(x)\sim\beta_G(y)$.

\underline{Surjectivity:} We only need to show that $\beta_G$ is surjective. For any symmetric acyclic orientation $\omega$ of $G$, we can find a symmetric linear extension
\[\omega_{-n}<\cdots <\omega_{-1}<\omega_{0}<\omega_1<\cdots<\omega_n\]
of all vertices $V(G)=[-n,n]$ that is compatible with $\omega$, and that $\omega_i+\omega_{-i}=0$ for all $i=0,1,\dots,n$. Then we choose real numbers
\[-0.5<x_{\omega_{-n}}<\cdots<x_{\omega_{-1}}<x_{\omega_0}=0<x_{\omega_1}<\cdots<x_{\omega_n}<0.5\]
and let $x=(x_1,\dots,x_n)$, so that $\beta_G(x)=\omega$.

\underline{Injectivity:} We need to show that for any two points $x,y$, if $\beta_G(x)=\beta_G(y)$ or $\beta_G(x)$ and $\beta_G(y)$ differ by a flip, then $x,y$ lie in the same chamber. If $\beta_G(x)=\beta_G(y)$, then the segment connecting $x$ and $y$ will not cross any hyperplanes. Therefore $x$ and $y$ lie in the same chamber. If $\beta_G(x)$ and $\beta_G(y)$ differ by a source-sink flip at $(v_i,v_{-i})$, assume $v_i$ is the source and $v_{-i}$ is the sink. Then the segment connecting $x$ and $y+\vec{e}_i=(y_1,\dots,y_{i-1},y_i+1,\dots,y_n)$ will not cross any hyperplanes. Therefore $x$ and $y$ lie in the same chamber.
\end{proof}

\subsection{Chambers of $\cB_\text{tor}(G)$ and the Bivariate Chromatic Polynomial of $G$}

According to \cref{thm:chamandchar}, the number of chambers $|\text{Cham}(\cB_\text{tor}(G))|$ is related to the toric characteristic polynomial $\chi_{\cB_\text{tor}(G)}$. The following lemma in \cite{Ehren_AffineToric_09} implies a relationship between $\chi_{\cB_\text{tor}(G)}$ and the bivariate chromatic polynomial $\chi_G$.

\begin{lemma}[{\cite[Theorem~3.7]{Ehren_AffineToric_09}}]\label{lm:finitefield}
Given a toric hyperplane arrangement $\cH_\text{tor}$, there exists infinitely many positive integers $q$ such that the toric characteristic polynomial evaluated at $q$ is given by the number of lattice points in $\left(\frac{1}{q}\Z\right)^n/\Z^n$ that
do not lie on any of the toric hyperplanes $H_i$, that is,
\[\chi_{\cH}(q)=\abs{\left(\frac{1}{q}\Z\right)^n\Big/\Z^n-\cH_\text{tor}}.\]
\end{lemma}

\begin{prop}\label{prop:charandchar}
For a symmetric graph $G$, the relationship between the bivariate chromatic polynomial $\chi_G(k,l)$ and the characteristic polynomial of $\cL(\cB_\text{tor}(G))$ is the following:
\[\chi_{\cB_\text{tor}(G)}(q)=\chi_G\left(\frac{q-2}{2},2\right).\]
\end{prop}

\begin{proof}
Assume $q=2k+2$ is even. Given any point in the lattice $\left(\frac{1}{q}\Z\right)^n/\Z^n$ that
do not lie on $\cB_\text{tor}(G)$
\[x=\left(\frac{x_1}{q},\dots,\frac{x_n}{q}\right)\in \left(\frac{1}{q}\Z\right)^n\Big/\Z^n-\cB_\text{tor}(G)\]
where $x_1,\dots,x_n$ are integers in $[-k,k+1]$, we can define a unique $(k,2)$-coloring $f$ of $G$ with colors $\{0,\pm1,\dots,\pm k\}\cup\{0'\}$ given by
\[f(v_i)=\begin{cases}
0' &\text{if }x_i=k+1,\\
x_i &\text{otherwise.}
\end{cases}\]
It is not hard to show that this coloring is indeed proper, and all proper colorings come from this way. Therefore, the number of such lattice points $x$ is equal to $\chi_G(k,2)$. On the other hand, by \cref{lm:finitefield} the same number is equal to $\chi_{\cB_\text{tor}(G)}(q)$ for infinitely many $q$. This concludes our proof.
\end{proof}

We are now in position to prove the main theorem.

\begingroup
\def\thetheorem{\ref{thm:main}}
\begin{theorem}
For a weakly connected symmetric graph $G$ with $2n+1$ vertices, we have
\[|\text{Acyc}(G)/\sim|=(-1)^n\chi_G(-1,2).\]
\end{theorem}
\addtocounter{theorem}{-1}
\endgroup

\begin{proof}
By \cref{thm:chamandacyc}, \cref{thm:chamandchar} and \cref{prop:charandchar}, we have
\[|\text{Acyc}(G)/\sim|=|\text{Cham }\cB_\text{tor}(G)|=(-1)^n\chi_{\cB_\text{tor}(G)}(0)=(-1)^n\chi_G(-1,2).\]
\end{proof}

\begin{remark}
Novik et al. \cite{Novik_SyzygiesOriented_00} proved that the number of chambers in a toric arrangement $|\text{Cham}(\cH_\text{tor})|$ is equal to the evaluation of the Tutte polynomial $T_\cH(1,0)$ if $\cH$ is a unimodular hyperplane arrangement (if the principal normal vectors of $\cH$ form a unimodular matrix). It turns out that $\cB(G)$ is one of the simplest non-unimodular hyperplane arrangement, and this work might shed light on how to generalize the results in \cite{Novik_SyzygiesOriented_00} to the general non-unimodular case. In \cite{DAdd_Arithmeticmatroids_11}, D’Adderio and Moci proposed using \emph{arithmetic Tutte polynomials} to calculate the number of chambers in a toric arrangement, but their method only works when $\cH$ is central, and cannot deal with hyperplanes such as $2x_i=1$ in our case. So far, it is unknown how are the arithmetic Tutte polynomials related to $\chi_G(k,l)$ in this paper.
\end{remark}

%% file: 5-Delcont.tex
\section{A second proof: Deletion-Contraction}\label{sec:delcont}

In this section, we present an alternative proof to the fact that the number of chambers in $\cB_\text{tor}(G)$ equals $(-1)^n\chi_G(-1,2)$ using deletion-contraction recurrences.

\begin{definition}
For a toric hyperplane arrangement $\cH$ of dimension $n$ and a toric hyperplane $H\in \cH$, the intersection $\cH\cap H$ is also a toric hyperplane arrangement of dimension $n-1$. We call $\cH\cap H$ a \emph{restriction}. If we take away the hyperplane $H$, the remaining hyperplane arrangement $\cH-H$ is called a \emph{deletion}.
\end{definition}

\begin{theorem}[Deletion-restriction of toric hyperplane arrangements]\label{thm:delres}
Given an essential toric hyperplane arrangement $\cH$ and a hyperplane $H\in\cH$,
\[|\text{Cham }\cH|=\begin{cases}
|\text{Cham }(\cH\cap H)| + |\text{Cham }(\cH-H)| &\text{if $\cH-H$ is essential},\\
|\text{Cham }(\cH\cap H)| &\text{otherwise}.
\end{cases}\]
\end{theorem}

\begin{proof}
Proof analogous to the proof of \cref{thm:chamandchar} (see \cite[Theorem~3.6]{Ehren_AffineToric_09}).
\end{proof}

\begin{definition}
Given a symmetric graph $G=(V,E)$, an edge $e=v_iv_j\in E$ is \emph{contractible} if $i,j\neq 0$ and $j\neq -i$. Given a contractible edge $e=v_iv_j$, denote the edge $-e=v_{-i}v_{-j}\in E$. The \emph{deletion} $G-e$ is the graph obtained by deleting edges $e$ and $-e$ from $G$. The \emph{contraction} $G/e$ is the graph obtained by contracting edges $e$ and $-e$ in $G$.
\end{definition}

\begin{theorem}[Deletion-contraction of colorings]\label{thm:delcont}
Given a symmetric graph $G$ and a contractible edge $e$, we have
\[\chi_G(k,l)=\chi_{G-e}(k,l)-\chi_{G/e}(k,l).\]
\end{theorem}
\begin{proof}
Assume $e=v_iv_j$. We only need to prove that the equality holds for any integers $k,l\geq 1$. In such case $\chi_{G-e}(k,l)$ counts the number of proper $(k,l)$-colorings of $G-e$. For each of these colorings, we can split them into two categories based whether $v_i$ and $v_j$ have the same color. If $v_i$ and $v_j$ have different colors, the number of such colorings is counted by $\chi_G(k,l)$; if $v_i$ and $v_j$ have the same color, we can contract $e=v_iv_j$ and the number of such colorings is counted by $\chi_{G/e}(k,l)$. Therefore, $\chi_{G-e}(k,l)=\chi_G(k,l)+\chi_{G/e}(k,l)$.
\end{proof}

Finally, \cref{thm:main} is a direct corollary of the following theorem.

\begin{theorem}\label{thm:mainext}
For any symmetric graph $G$ with $2n+1$ vertices, we have
\[(-1)^n\chi_G(-1,2)=\begin{cases}
|\text{Cham}(\cB_\text{tor}(G))| &\text{if $G$ is weakly connected},\\
0 &\text{otherwise}.
\end{cases}\]
\end{theorem}

\begin{proof}
Use induction. In the base case, $G$ has the form of the following graph in \cref{fig:5}. In such case, the left hand side is $\chi_G(k,l)=(2k+l)^a\cdot(2k+l-1)^b\cdot(2k)^c$. The right hand side when $a=0$ is the toric hyperplane arrangement of toric hyperplanes $\{x_{j_p}=0\}\cup \{x_{l_q}=0\}\cup \{x_{l_q}=1/2\}$ in the tori. It cuts the tori into $2^c$ chambers. Easy to check that the number matches. The induction step follows from \cref{thm:delres} and \cref{thm:delcont}.
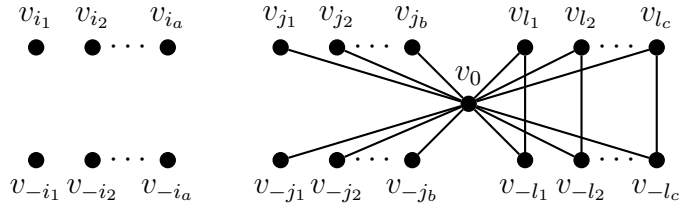
\begin{figure}[ht!]
\centering
\begin{tikzpicture}
    \SetFancyGraph
    \SetGraphUnit{1.5}
    \Vertex[L=$v_{i_1}$,Lpos=above]{1a}
    \SO[L=$v_{-i_1}$,Lpos=below](1a){-1a}
    \EA[L=$v_{i_2}$,Lpos=above,unit=0.75](1a){2a}
    \SO[L=$v_{-i_2}$,Lpos=below](2a){-2a}
    \begin{scope}[VertexStyle/.style = {draw=none}]
    \SetVertexLabelIn
    \EA[L=$\cdots$,unit=0.5](2a){3a}
    \SO[L=$\cdots$](3a){-3a}
    \end{scope}
    \EA[L=$v_{i_a}$,Lpos=above,unit=0.5](3a){4a}
    \SO[L=$v_{-i_a}$,Lpos=below](4a){-4a}
    
    \EA[L=$v_{j_1}$,Lpos=above](4a){1b}
    \SO[L=$v_{-j_1}$,Lpos=below](1b){-1b}
    \EA[L=$v_{j_2}$,Lpos=above,unit=0.75](1b){2b}
    \SO[L=$v_{-j_2}$,Lpos=below](2b){-2b}
    \begin{scope}[VertexStyle/.style = {draw=none}]
    \SetVertexLabelIn
    \EA[L=$\cdots$,unit=0.5](2b){3b}
    \SO[L=$\cdots$](3b){-3b}
    \end{scope}
    \EA[L=$v_{j_b}$,Lpos=above,unit=0.5](3b){4b}
    \SO[L=$v_{-j_b}$,Lpos=below](4b){-4b}
    \SOEA[L=$v_{0}$,Lpos=above,unit=0.75](4b){0}
    \Edges(1b,0,-1b)
    \Edges(2b,0,-2b)
    \Edges(4b,0,-4b)
    
    \NOEA[L=$v_{l_1}$,Lpos=above,unit=0.75](0){1c}
    \SO[L=$v_{-l_1}$,Lpos=below](1c){-1c}
    \EA[L=$v_{l_2}$,Lpos=above,unit=0.75](1c){2c}
    \SO[L=$v_{-l_2}$,Lpos=below](2c){-2c}
    \begin{scope}[VertexStyle/.style = {draw=none}]
    \SetVertexLabelIn
    \EA[L=$\cdots$,unit=0.5](2c){3c}
    \SO[L=$\cdots$](3c){-3c}
    \end{scope}
    \EA[L=$v_{l_c}$,Lpos=above,unit=0.5](3c){4c}
    \SO[L=$v_{-l_c}$,Lpos=below](4c){-4c}
    \Edges(0,1c,-1c,0)
    \Edges(0,2c,-2c,0)
    \Edges(0,4c,-4c,0)
\end{tikzpicture}
\caption{Base case of the induction}\label{fig:5}
\end{figure}
\end{proof}

%% file: 6-Directly-1.tex
\section{A third proof: Direct Induction}\label{sec:directly}

In this section, we will present a direct proof of \cref{thm:main} without using the geometric interpretations. We will start with the following theorem that evaluates $(-1)^n\chi_G(-1,2)$.

\begin{definition}
In a symmetric graph $G$ with $2n+1$ vertices, a set $I\subseteq\{1,2,\dots,n\}$ is a \emph{independent set} if the vertices $V(I)=\{v_i:i\in I\text{ or }-i\in I\}$ is an independent set in $G$. Let $G-I$ denote the symmetric graph obtained by deleting all vertices in $V(S)$ from $G$. 
\end{definition}

\begin{theorem}\label{thm:altersum}
Given a symmetric graph $G$ with $2n+1$ vertices, we have
\[(-1)^n\chi_G(-1,2)=\sum_{I\subseteq[n]\text{ indep.}}(-1)^{|I|}|\text{Acyc}(G-I)|.\]
\end{theorem}

\begin{proof}
Given a proper $(k,2)$-coloring of $G$ using colors $\{0,\pm1,\dots,\pm k\}\cup\{0'\}$, let $I$ be the independent set of vertices colored by $0'$. Then the number $\chi_G(k,2)$ is equal to the number of tuples $(I,f)$, where $I$ is an independent set of $G$ and $f$ is a proper $k$-coloring of the subgraph $G-I$. This implies that 
\[\chi_G(k,2)=\sum_{I\subseteq[n]\text{ indep.}}\chi_{G-I}(k,1).\]
The theorem then follows from letting $k=-1$ and applying \cref{thm:zas}, which says that $|\text{Acyc(G)}|=(-1)^n\chi_G(-1,1)$.
\end{proof}

\begin{example}
We show an example of \cref{thm:altersum} using the graph $G$ in \cref{fig:1b}. The independent sets of $G$ are $I=\emptyset$, $\{1\}$ and $\{2\}$. Therefore, \cref{thm:altersum} tells us that the evaluation of the bivariate chromatic polynomial $(-1)^n\chi_G(-1,2)$ is equal to the alternating sum in \cref{fig:6}. The number of acyclic orientations of the three subgraphs on the right hand side in \cref{fig:6} are $6,2$ and $1$ respectively (see \cref{fig:2} for all $6$ orientations of the first subgraph), and the theorem holds since $3=6-2-1$.
\begin{figure}[ht!]
\centering
\begin{tikzpicture}
    \SetFancyGraph
    \SetGraphUnit{1}
    \begin{scope}[VertexStyle/.style = {draw=none}]
    \SetVertexLabelIn
    \Vertex[L={$(-1)^n\chi(-1,2)$}]{start}
    \EA[L={$=$},unit=1.7](start){=}
    \end{scope}
    \EA[L=$v_1$,Lpos=left,unit=1.5](=){1}
    \EA[L=$v_0$,Lpos=left](1){0}
    \NO[L=$v_2$,Lpos=above](0){2}
    \SO[L=$v_{-2}$,Lpos=below](0){-2}
    \EA[L=$v_{-1}$,Lpos=right](0){-1}
    \Edges(1,2,-1)
    \Edges(-2,0,2)
    \Edges(1,-2,-1)
    
    \begin{scope}[VertexStyle/.style = {draw=none}]
    \SetVertexLabelIn
    \EA[L={$-$},unit=1.5](-1){+one}
    \end{scope}

    \EA[L=$v_0$,Lpos=left,unit=1.5](+one){0a}
    \SO[L=$v_{-2}$,Lpos=below](0a){-2a}
    \NO[L=$v_{2}$,Lpos=above](0a){2a}
    \Edges(-2a,0a,2a)
    
    \begin{scope}[VertexStyle/.style = {draw=none}]
    \SetVertexLabelIn
    \EA[L={$-$},unit=1.2](0a){+two}
    \end{scope}

    \EA[L=$v_1$,Lpos=above,unit=1.2](+two){1b}
    \EA[L=$v_{0}$,Lpos=above](1b){0b}
    \EA[L=$v_{-1}$,Lpos=above](0b){-1b}
\end{tikzpicture}
\caption{An example of \cref{thm:altersum} using \cref{fig:1b}}\label{fig:6}
\end{figure}
\end{example}

In order to prove the main theorem, what is left is to show that the right hand side of \cref{thm:altersum} satisfies
\begin{equation}\label{eq:acycspecial}
    |\text{Acyc}(G)/\sim|=\sum_{I\subseteq[n]\text{ indep.}}(-1)^{|I|}|\text{Acyc}(G-I)|.
\end{equation}
The strategy is to use induction. We would like to split each term $(-1)^{|I|}|\text{Acyc}(G-I)|$ in the alternating sum into two categories based on whether $i\in I$ or $i\notin I$ for some fixed index $i\in[n]$. If $i\in I$, then the sum reduces to the case of $G-\{i\}$. If $i\notin I$, the sum reduces to a version of $G$ where $i$ is ``frozen''. Namely, $i$ is not allowed to be picked out into the independent set $I$. This motivates the following definition.

\begin{definition}
A \emph{frozen symmetric graph} is a pair $\Gamma=(G,S)$ of symmetric graph $G$ and a set $S\subseteq[n]$. Vertices in $V(S)=\{v_i:i\in S\}\cup\{v_{-i}:i\in S\}$ are called \emph{frozen vertices}; other vertices are called \emph{free vertices}. Two acyclic orientations $\omega$ and $\omega'$ of $G$ differ by an \emph{$S$-flip} if one can reverse all arrows adjacent to a \emph{free} source-sink pair $(v_i,v_{-i})$ where $i\not\in S$ in $\omega$ to obtain $\omega'$. In other words, we are not allowed to flip the frozen vertices. The transitive closure of the $S$-flip operation generates an equivalence relation on the set of all symmetric acyclic orientation $\text{Acyc}(G)$ denoted by $\sim_S$.
\end{definition}

Also, we say $\Gamma=(G,S)$ is \emph{weakly connected} if vertex $v_i$ is connected to $v_{-i}$ in $G$, or connected to a frozen vertex, for all $i\neq 0$. This definition is compatible with the definition of weakly-connectedness for regular symmetric graphs in \cref{def:weaklyconn} (when $S=\emptyset$). We will assume the frozen symmetric graph $\Gamma$ is weakly-connected in the rest of the paper. The following theorem is a generalization of \cref{eq:acycspecial}.

\begin{theorem}\label{thm:frozen}
Given a frozen symmetric graph $\Gamma=(G,S)$, we have
\[
    |\text{Acyc}(G)/{\sim_S}|=\sum_{\substack{I\subseteq [n]\text{ indep.}\\I\cap S=\emptyset}}(-1)^{|I|}|\text{Acyc}(G-I)|.
\]
\end{theorem}

The proof of \cref{thm:frozen} relies on the following definitions and a key lemma.

\begin{definition}
Given a frozen symmetric graph $\Gamma=(G,S)$ and index $i\not\in S$ where $v_i,v_{-i}$ are not adjacent, denote the graph $\Gamma_i:=(G,S\cup\{i\})$ called the \emph{freezing} of $G$ at $i$. Let $G-i$ be the subgraph  of $G$ obtained by deleting vertices $v_i$ and $v_{-i}$, and let $N(i)\subseteq[n]$ denote the set of indices $j$ where $v_j$ is a neighbour of $v_i$ or $v_{-i}$. Denote the graph $\Gamma-i:=(G-i,S\cup N(i))$ called the \emph{deletion} of $G$ at $i$.
\end{definition}

\begin{remark}
The freezing $\Gamma_i$ is obtained by simply ``freezing'' the vertices $v_i$ and $v_{-i}$ in $\Gamma$. To obtain the deletion $\Gamma-i$, one deletes vertices $v_i$ and $v_{-i}$, and then ``freezes'' all neighbours of them.
\end{remark}

\begin{example}
\cref{fig:7} shows an example of freezing and deletion of a frozen symmetric graph $\Gamma=(G,\{2,4,5\})$ at index $4$. Frozen vertices are labeled by squares.
\begin{figure}[ht!]
\centering
\begin{subfigure}[c]{.3\linewidth}
\centering
\begin{tikzpicture}
    \SetFancyGraph
    \SetGraphUnit{0.9}
    \Vertex[L=$v_0$,Lpos=right]{0}
    \NO[L=$v_3$,Lpos=above](0){3}
    \WE[L=$v_2$,Lpos=above](3){2}
    \WE[L=$v_1$,Lpos=above](2){1}
    \EA[L=$v_4$,Lpos=above](3){4}
    \EA[L=$v_5$,Lpos=above](4){5}
    
    \SO[L=$v_{-3}$,Lpos=below](0){-3}
    \WE[L=$v_{-2}$,Lpos=below](-3){-2}
    \WE[L=$v_{-1}$,Lpos=below](-2){-1}
    \EA[L=$v_{-4}$,Lpos=below](-3){-4}
    \EA[L=$v_{-5}$,Lpos=below](-4){-5}
    
    \Edges(0,3,4,5,0)
    \Edges(0,-3,-4,-5,0)
    \Edges(1,0,-1)
    \Edges(2,0,-2)
    
    \begin{scope}
    \tikzset{VertexStyle/.style = {
        shape = rectangle,
        fill= none,
        inner sep = 0pt,
        outer sep = 0pt,
        minimum size = 10pt,
    draw}}
    \NO[unit=0,NoLabel](1){x}
    \NO[unit=0,NoLabel](3){y}
    \NO[unit=0,NoLabel](-1){-x}
    \NO[unit=0,NoLabel](-3){-y}
    \end{scope}
\end{tikzpicture}
\caption{Frozen Sym. Graph}
\end{subfigure}
\begin{subfigure}[c]{.3\linewidth}
\centering
\begin{tikzpicture}
    \SetFancyGraph
    \SetGraphUnit{0.9}
    \Vertex[L=$v_0$,Lpos=right]{0}
    \NO[L=$v_3$,Lpos=above](0){3}
    \WE[L=$v_2$,Lpos=above](3){2}
    \WE[L=$v_1$,Lpos=above](2){1}
    \EA[L=$v_4$,Lpos=above](3){4}
    \EA[L=$v_5$,Lpos=above](4){5}
    
    \SO[L=$v_{-3}$,Lpos=below](0){-3}
    \WE[L=$v_{-2}$,Lpos=below](-3){-2}
    \WE[L=$v_{-1}$,Lpos=below](-2){-1}
    \EA[L=$v_{-4}$,Lpos=below](-3){-4}
    \EA[L=$v_{-5}$,Lpos=below](-4){-5}
    
    \Edges(0,3,4,5,0)
    \Edges(0,-3,-4,-5,0)
    \Edges(1,0,-1)
    \Edges(2,0,-2)
    
    \begin{scope}
    \tikzset{VertexStyle/.style = {
        shape = rectangle,
        fill= none,
        inner sep = 0pt,
        outer sep = 0pt,
        minimum size = 10pt,
    draw}}
    \NO[unit=0,NoLabel](1){x}
    \NO[unit=0,NoLabel](3){y}
    \NO[unit=0,NoLabel](4){z}
    \NO[unit=0,NoLabel](-1){-x}
    \NO[unit=0,NoLabel](-3){-y}
    \NO[unit=0,NoLabel](-4){-z}
    \end{scope}
\end{tikzpicture}
\caption{Freezing at $4$}
\end{subfigure}
\begin{subfigure}[c]{.3\linewidth}
\centering
\begin{tikzpicture}
    \SetFancyGraph
    \SetGraphUnit{0.9}
    \Vertex[L=$v_0$,Lpos=right]{0}
    \NO[L=$v_3$,Lpos=above](0){3}
    \WE[L=$v_2$,Lpos=above](3){2}
    \WE[L=$v_1$,Lpos=above](2){1}
    \EA[L=$v_5$,Lpos=above,unit=1.8](3){5}
    
    \SO[L=$v_{-3}$,Lpos=below](0){-3}
    \WE[L=$v_{-2}$,Lpos=below](-3){-2}
    \WE[L=$v_{-1}$,Lpos=below](-2){-1}
    \EA[L=$v_{-5}$,Lpos=below,unit=1.8](-3){-5}
    
    \Edges(3,0,5)
    \Edges(-3,0,-5)
    \Edges(1,0,-1)
    \Edges(2,0,-2)
    
    \begin{scope}
    \tikzset{VertexStyle/.style = {
        shape = rectangle,
        fill= none,
        inner sep = 0pt,
        outer sep = 0pt,
        minimum size = 10pt,
    draw}}
    \NO[unit=0,NoLabel](1){x}
    \NO[unit=0,NoLabel](3){y}
    \NO[unit=0,NoLabel](5){z}
    \NO[unit=0,NoLabel](-1){-x}
    \NO[unit=0,NoLabel](-3){-y}
    \NO[unit=0,NoLabel](-5){-z}
    \end{scope}
\end{tikzpicture}
\caption{Deletion at $4$}
\end{subfigure}
\caption{Example of freezing and deletion of a frozen symmetric graph. Frozen vertices labeled by squares.}\label{fig:7}
\end{figure}
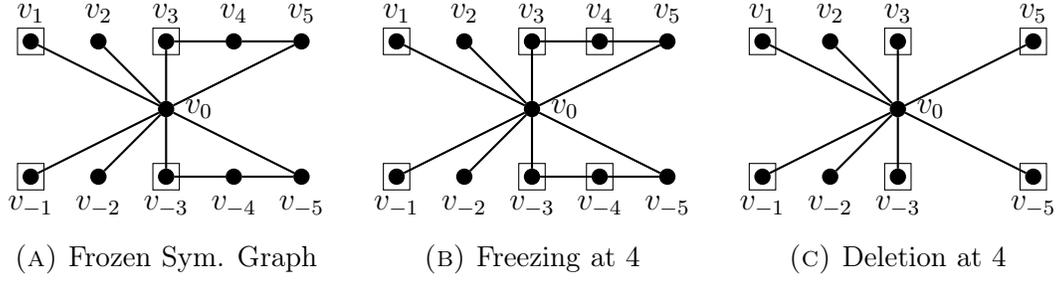
\end{example}

Now we present the key lemma that is crucial to the proof of \cref{thm:frozen}.

\begin{lemma}[Freezing-Deletion of acyclic orientations]\label{lm:frozen}
Given a frozen symmetric graph $\Gamma=(G,S)$ and $i\not\in S$ where $v_{i},v_{-i}$ are not adjacent, denote $\Gamma_i=(G,S\cup\{i\})$ and $\Gamma-i=(G-i,S\cup N(i))$. We have
\[|\text{Acyc}(G)/{\sim_S}|=|\text{Acyc}(G)/{\sim_{S\cup\{i\}}}|-|\text{Acyc}(G-i)/{\sim_{S\cup N(i)}}|.\]
\end{lemma}

We will postpone the proof for \cref{lm:frozen} to the last subsection. Before that, we will show how \cref{lm:frozen} implies \cref{thm:frozen}, and thus the main result.

\begin{proof}[Proof of \cref{thm:frozen}]
We prove by induction on $(n,|S|)$ where the second term $|S|$ goes in reverse order. In the base case, $S=\{i\in [n]:v_{i}\text{ and }v_{-i}\text{ are not adjacent}\}$, and both sides are equal to $|\text{Acyc}(G)|$. For the induction step, pick a random index $i\not\in S$ and $v_{i},v_{-i}$ not adjacent. We split the alternating sum on the right hand side into two cases: $i\in I$ or $i\notin I$.
\begin{align*}
    &\sum_{\substack{I\subseteq [n]\text{ indep.}\\I\cap S=\emptyset\\i\notin I}}(-1)^{|I|}|\text{Acyc}(G-I)| \quad+\quad \sum_{\substack{I\subseteq [n]\text{ indep.}\\I\cap S=\emptyset\\i\in I}}(-1)^{|I|}|\text{Acyc}(G-I)|\\
    =&\sum_{\substack{I\subseteq [n]\text{ indep.}\\I\cap (S\cup\{i\})=\emptyset}}(-1)^{|I|}|\text{Acyc}(G-I)| \quad-\quad \sum_{\substack{I'\subseteq [n]-i\text{ indep.}\\I'\cap S=\emptyset\\I'\cap N(i) =\emptyset}}(-1)^{|I'|}|\text{Acyc}(G-i-I')|.
\end{align*}
When $i\notin I$, it reduces to the case of $\Gamma_i$, and when $i\in I$, it reduces to the case $\Gamma-i$. The induction step then follows from \cref{lm:frozen}.
\end{proof}

Specifically, when $S=\emptyset$ we get the following corollary.

\begin{corollary}\label{cor:frozen}
Given a symmetric graph $G$, we have
\[
|\text{Acyc}(G)/{\sim}|=\sum_{I\subseteq [n]\text{ indep.}}(-1)^{|I|}|\text{Acyc}(G-I)|.
\]
\end{corollary}

\cref{thm:main} then follows from \cref{thm:altersum} and \cref{cor:frozen}.

%% file: 6-Directly-2.tex
\subsection{Proof of \cref{lm:frozen}}

For a frozen symmetric graph $\Gamma=(G,S)$, consider a graph $\cG(\Gamma)$ whose vertex set is the set of all possible acyclic orientations of $G$, and there is an edge between two orientations $\omega$ and $\omega'$ of $G$ if they differ by an $S$-flip. Then the number $|\text{Acyc(G)}/\sim_S|$ is the number of connected components of $\cG(\Gamma)$.

Now for $\Gamma=(G,S)$ and $j\not\in S$, $\cG(\Gamma_j)$ has the same vertex set as $\cG(\Gamma)$, so we can view $\cG(\Gamma_j)$ as an induced subgraph of $\cG(\Gamma)$ with the same vertex set and fewer edges. For $\cG(\Gamma-j)$, we can add $v_j$ as a source and $v_{-j}$ as a sink to each vertex of $\cG(\Gamma-j)$. Then we can identify $\cG(\Gamma-j)$ as a induced subgraph of $\cG(\Gamma)$ in this case as well, but with both fewer vertices and fewer edges. Denote $\#\text{Comp}(\cG)$ as the number of connected components in a graph $\cG$. Our goal is to show the following
\[\#\text{Comp}(\cG(\Gamma))=\#\text{Comp}(\cG(\Gamma_j))-\#\text{Comp}(\cG(\Gamma-j)).\]

Now since $\cG(\Gamma_j)$ and $\cG(\Gamma-j)$ are both induced subgraphs of $\cG(\Gamma)$, we can restrict our attention to a single connected component $\cC\subseteq \cG(\Gamma)$. We only need to show that
\begin{equation}\label{eq:goal}
    \#\text{Comp}(\cG(\Gamma_j)\cap \cC)=\#\text{Comp}(\cG(\Gamma-j)\cap \cC)+1.
\end{equation}

We would like to define some ``weight function'' for each orientation $\omega\in\cC$, such that it satisfies some nice properties.

\begin{prop}\label{prop:height}
There exists a family of \emph{height functions} $\{h_\omega:V(G)\to\R\}_{\omega\in\cC}$ that satisfies the following properties:
\begin{enumerate}
    \item(Frozen zeros) For any frozen vertex $v\in V(S)$, $h_\omega(v)=0$.
    \item(Symmetric) $h_\omega(v_i)+h_\omega(v_{-i})=0$ for any $i\in [n]$.
    \item(Compatible with orientation) If $u\to v$ in orientation $\omega$, then $0\leq h_\omega(u)-h_\omega(v)<1$, and $h_\omega(u)=h_\omega(v)$ iff they are both zeros.
    \item(Compatible with flips) For any adjacent $\omega, \omega'\in \cC$ that differ by flipping $(v_i,v_{-i})$ from a source/sink pair to a sink/source pair in $\omega$, we have 
    \[h_{\omega}(v)-h_{\omega'}(v)=\begin{cases}
    +1, &\text{if } v=v_i,\\
    -1, &\text{if } v=v_{-i},\\
    0, &\text{otherwise}.
    \end{cases}\]
    \item(Integral) For any $\omega, \omega'\in \cC$ and vertex $v$, $h_{\omega}(v)-h_{\omega'}(v)$ is an integer.
    \item(Uniqueness) If $h_\omega=h_{\omega'}$, then $\omega=\omega'$.
\end{enumerate}
\end{prop}
\begin{proof}
Define $C\subset G$ to be a closed cycle of $\Gamma=(G,S)$ if $C$ is a (directed) loop or $C$ is a (directed) path that starts and ends at two frozen vertices. For any closed cycle $C\subseteq G$ and an orientation $\omega$ of $G$, define the \emph{circulation} of $\omega$ around $C$ as $c_\omega(C)=|C_\omega^+|-|C_\omega^-|$, where $C_\omega^+$ is the set of forward edges in $C$ and $C_\omega^-$ is the set of backward edges in $C$. Note that flipping at a source-sink pair will not change the circulation of any closed cycle. Therefore, $c=c_\omega$ is fixed across all $\omega\in\cC$.

To proceed with the proof, we need an extra ingredient: a real-valued function $f_\cC$ on the set of directed edges of $G$ (in both directions, denoted as $\vec{E}(G)$). To define $f_\cC$, we first define $f_\omega:\vec{E}(G)\to \R$ for any $\omega\in\cC$ as follows:
\[f_\omega(\vec{e})=\begin{cases}
+1, &\text{if }\vec{e}\in\omega,\\
0, &\text{if }\vec{e}\not\in\omega.
\end{cases}\]
It is not hard to check that $f_\omega$ satisfies the following four properties:
\begin{enumerate}[label=(\alph*)]
    \item $0\leq f_\omega(\vec{e})\leq 1$;
    \item $f_\omega(\vec{e})+f_\omega(-\vec{e})=1$;
    \item $f_\omega(v_i\to v_k)=f_\omega(v_{-k}\to v_{-i})$;
    \item For any closed cycle $C\subset G$, $\sum_{\vec{e}\in C}f_\omega(\vec{e})=\frac{1}{2}(|C|+c(C))$.
\end{enumerate}

Now we define $f_\cC(\vec{e})$ to be the average of all $f_\omega(\vec{e})$ ranging over all $\omega\in \cC$. Therefore $f_\cC$ satisfies the four properties above as well. Now we claim that there is a unique function $h_\omega:V(G)\to \R$ for any $\omega\in \cC$ such that it satisfies properties (1), (2), and for any $\vec{e}=u\to v$,
\begin{equation}\label{eq:case}
h_\omega(u)-h_\omega(v)=\begin{cases}
    1-f_\cC(\vec{e}), &\text{if }\vec{e}\in\omega,\\
    -f_\cC(\vec{e}), &\text{if }\vec{e}\not\in\omega.
\end{cases}
\end{equation}

\underline{Uniqueness:} Since $\Gamma$ is simply connected, there exists a directed path $P=\vec{e_1},\dots,\vec{e_k}$ either between $v_i$ and $v_{-i}$, or between $v_i$ and a frozen vertex $u$. In the first case, $h_\omega(v_i)-h_\omega(v_{-i})=t-f_\cC(\vec{e_1})-\cdots-f_\cC(\vec{e_k})$ where $t$ is the number of $\vec{e_i}$ such that $\vec{e_i}\in\omega$. Combine with the fact that $h_\omega(v_i)+h_\omega(v_{-i})=0$, we know that $h_\omega(v_i)$ and $h_\omega(v_{-i})$ are uniquely determined. In the second case, $h_\omega(v_i)=h_\omega(v_i)-h_\omega(u)=t-f_\cC(\vec{e_1})-\cdots-f_\cC(\vec{e_k})$ for the same $t$, so again $h_\omega(v_{i})$ is uniquely determined.

\underline{Existence:} We only need to check that the definition of $h_\omega$ is self-consistent on every closed cycle $C$. Let $\vec{e_1},\dots,\vec{e_k}$ be edges on the closed cycle, then we only need to show that $0=t-f_\cC(\vec{e_1})-\cdots-f_\cC(\vec{e_k})$, where $t$ is the number of forward edges on $C$, which is also equal to $\sum_{i=1}^kf_\cC(\vec{e_k})$ by property (d) of $f_\cC$.

Finally, we show that $h_\omega$ satisfies properties (3), (4) and (5). For (3), \cref{eq:case} guarantees that $0\leq h_\omega(u)-h_\omega(v)\leq 1$ for $\vec{e}=u\to v\in \omega$. If $h_\omega(u)-h_\omega(v)=1$, then $f_\cC(\vec{e})=0$, which implies that $\vec{e}\notin \omega'$ for all $\omega'\in\cC$, which contradicts with the case $\omega'=\omega$ itself. If $h_\omega(u)-h_\omega(v)=0$, then $f_\cC(\vec{e})=1$, which implies that $\vec{e}\in \omega'$ for all $\omega'\in\cC$. In other words, the direction of the edge $\vec{e}$ has never changed in all orientations in $\cC$, so neither $u$ nor $v$ has ever been flipped in $\cC$, otherwise it will change the direction of $\vec{e}$. As a result, we can assume $u$ and $v$ are frozen vertices from the start, so $f_\omega(u)=f_\omega(v)=0$.

For (4), for any vertex $v_j\neq v_i,v_{-i}$, since we never flipped $v_j$ or $v_{-j}$, the path $P$ between $v_j$ and $v_{-j}$ (or a frozen vertex $u$) has the same circulation in $\omega$ and $\omega'$ (see the Uniqueness section). As a result, $h_\omega(v_j)=h_{\omega'}(v_j)$. Now assume $v_k$ is adjacent to $v_i$, and by \cref{eq:case}, since we flipped $v_i$, the evaluation of $h_\omega(v_i)$ changes by $1$. The same for $h_\omega(v_{-i})$. (5) is obvious from (4).

For (6), assume $h_\omega=h_{\omega'}$. Property (3) tells us that the orientation of most edges $e=(u,v)$ are known by comparing $h_\omega(u)$ and $h_\omega(v)$, unless $h_\omega(u)=h_\omega(v)=0$. In that case, the direction of the edge $\vec{e}$ has never changed in all orientations in $\cC$, which implies that $\omega=\omega'$.
\end{proof}

\begin{prop}\label{prop:path}
For any $\omega,\omega'\in \cC$, there exists a path $\omega=\omega_0,\omega_1,\cdots,\omega_n=\omega'$ in $\cC$ such that only the vertices in the set $\{v:h_\omega(v)\neq h_{\omega'}(v)\}$ are flipped.
\end{prop}
\begin{proof}
We induct on the value $\Delta(\omega,\omega'):=\sum_{v\in V(G)}|h_\omega(v)-h_{\omega'}(v)|$. This value is an integer because of property (5). When $\Delta(\omega,\omega')=0$ we have $h_\omega=h_{\omega'}$ so $\omega=\omega'$. Assume $\Delta(\omega,\omega')>0$. Assume $v$ is the vertex satisfying $h_\omega(v)\neq h_{\omega'}(v)$ and maximized $\max\{h_\omega(v), h_{\omega'}(v)\}$. WLOG assume $h_\omega(v)>h_{\omega'}(v)$. Obviously $h_\omega(v)>0$ (otherwise replace $v=v_i$ with $v_{-i}$). We claim that $v$ is a source in $\omega$. Otherwise there exists some $u\to v$ in $\omega$. As a result $h_\omega(u)>h_\omega(v)$. By maximality, we must have $h_\omega(u)=h_{\omega'}(u)$, but then $h_{\omega'}(u)-h_{\omega'}(v)>h_\omega(v)-h_{\omega'}(v)\geq 1$ by (5), which contradicts with (3).

Now we can flip $v$ in $\omega$ to obtain $\omega''$, and the value of $\Delta(\omega'',\omega')$ decreases. By induction hypothesis, there is a path between $\omega'',\omega'$, which concludes the proof.
\end{proof}

\begin{corollary}\label{cor:gamma1}
For any $\omega,\omega'\in \cG(\Gamma_j)\cap\cC$, they belong to the same connected component in $\text{Comp}(\cG(\Gamma_j)\cap\cC)$ if and only if $h_\omega(v_j)=h_{\omega'}(v_j)$.
\end{corollary}
\begin{proof}
Because of \cref{prop:path}, we can find a path from $\omega$ to $\omega'$ that avoids flipping $v_j$, so the path lives in $\cG(\Gamma_j)\cap\cC$.
\end{proof}

\begin{corollary}\label{cor:gamma2}
For any $\omega,\omega'\in \cG(\Gamma-j)\cap\cC$, they belong to the same connected component in $\text{Comp}(\cG(\Gamma-j)\cap\cC)$ if and only if $h_\omega(v_j)=h_{\omega'}(v_j)$.
\end{corollary}
\begin{proof}
$h_\omega(v_j)=h_{\omega'}(v_j)$ implies that $h_\omega(u)=h_{\omega'}(u)$ for every $u\in N(j)$. Because of \cref{prop:path}, we can find a path from $\omega$ to $\omega'$ that avoids flipping $v_j$ and all vertices in $N(j)$, so the path lives in $\cG(\Gamma-j)\cap\cC$.
\end{proof}

\begin{proof}[Proof of \cref{lm:frozen}]
Assume $h_\omega(v_j)$ can hold $d$ different values $N_1<N_2<\cdots<N_d$ over all $\omega\in \cC$. Then by \cref{cor:gamma1} the number of connected components $\#\text{Comp}(\cG(\Gamma_j)\cap\cC)=d$. Now we claim that $h_\omega(v_j)$ can hold $d-1$ different values $N_2,\dots,N_d$ over all $\omega\in \cG(\Gamma-j)\cap\cC$. First, $h_\omega(v_j)\neq N_1$. Otherwise, we can flip $(v_j,v_{-j})$ to get $\omega'$, but $h_{\omega'}(v_j)=N_1-1<N_1$, which contradicts the fact that $N_1$ is the minimal possible value. On the other hand, for any $2\leq i\leq d$, there exists a path $\omega_0,\dots,\omega_n\in\cC$ such that $h_{\omega_0}(v_j)=N_i$ and $h_{\omega_n}(v_j)=N_1$. Find the last orientation $\omega_k$ in the path such that $h_{\omega_k}(v_j)=N_i$. Then $v_j$ must be a source in $\omega_k$, which implies that $\omega_k\in \cG(\Gamma-j)\cap\cC$. This concludes the proof of the claim. By \cref{cor:gamma2}, we have $\#\text{Comp}(\cG(\Gamma_j)\cap\cC)=d-1$. Therefore, \cref{eq:goal} holds.
\end{proof}

\begin{remark}
The height function $h_\omega$ in \cref{prop:height} is inspired by James Propp's construction of height functions for ordinary type A graphs in \cite{Propp_Latticestructure_02}. In Propp's paper, he used the height function to show that there is a poset structure on each equivalence class of $\text{Acyc}(G)/\sim$, which is defined by $\omega \gtrdot \omega'$ if we flip source $i(\neq u)$ to sink $i$ ($u$ is a given fixed vertex). Moreover, the poset is a lattice. As a result, each of the lattice has a unique maximum element: acyclic orientation of $G$ with a unique sink $u$. This maximum element gives us a \emph{representative object} for each equivalence class in $\text{Acyc}(G)/\sim$ in the type A case, which helps to set up concrete bijections. However, all the methods we just described do not work for type B symmetric graphs, because there is not an obvious poset structure, and even if we force a poset structure, the poset will not be a lattice and will not have a unique representative object. Therefore, it will be interesting to know if one can find a well-defined representative object through some other methods for each equivalence class in $\text{Acyc}(G)/\sim$ in the type B case, in order to have a bijective proof of \cref{lm:frozen}.
\end{remark}

\begin{remark}
Another open question raised by Alex Postnikov is, whether we can find a Lie theoretic analog of \cref{thm:main}. In other words, can we find a way to prove \cref{thm:main}  for both ordinary graphs and symmetric graphs simultaneously, while potentially generalizing to other Lie types.
\end{remark}